\documentclass[russian, oneside, 12pt, a4paper]{article}

\usepackage{amsthm}
\usepackage{amsmath}
\usepackage{amssymb}

\usepackage[utf8]{inputenc} 
\usepackage[T2A]{fontenc} 
\usepackage{cmap} 
\usepackage[russian]{babel} 
\usepackage{xcolor}
\usepackage{hyperref}
\definecolor{linkcolor}{HTML}{799B03} 
\definecolor{urlcolor}{HTML}{799B03} 
\hypersetup{linkcolor=linkcolor,urlcolor=urlcolor,colorlinks=true}

\theoremstyle{plain}
\newtheorem{theorem}{Теорема}
\newtheorem{statement}[theorem]{Утверждение}
\newtheorem{lemma}[theorem]{Лемма}
\newtheorem{hypothesis}[theorem]{Гипотеза}

\theoremstyle{definition}
\newtheorem*{definition}{Определение}
\newtheorem*{notation}{Обозначение}
\newtheorem*{note}{Замечание}
\newtheorem{question}[theorem]{Открытый вопрос}

\usepackage[figure,noend, noline]{algorithm2e}

\usepackage{graphicx}

\makeatletter
\renewcommand{\@makecaption}[2]{%
\vspace{\abovecaptionskip}%
\sbox{\@tempboxa}{#1. #2}
\ifdim \wd\@tempboxa >\hsize
   #1. #2\par
\else
   \global\@minipagefalse
   \hbox to \hsize {\hfil #1. #2\hfil}%
\fi
\vspace{\belowcaptionskip}}
\makeatother

\begin{document}

    \title{О дружественности деревьев}
    \author{Колодзей Дарья}

    \maketitle

    \begin{abstract}
        Впервые понятие дружественности деревьев появилось в решении гипотезы Ландо о пересечении двумерных сфер в трёхмерном пространстве.

{\it Дерево дружественно простому пути}, если рёбра этого дерева можно пронумеровать так, что для любых $k, s$ путь между рёбрами $k$ и $k + 1$ содержит либо оба ребра $k + 2s, k + 2s + 1$, либо ни одного из этих рёбер.

{\bf Теорема}. {\it Пусть в дереве существует путь, содержащий все вершины этого дерева степени $3$ и более. Тогда это дерево дружественно простому пути}.

В работе доказано ещё одно достаточное условие дружественности дерева и простого пути, а также критерий дружественности двух деревьев, одно из которых имеет диаметр $3$.
    \end{abstract}

    \section{Введение}
        \subsection{Формулировки результатов}
            \graphicspath{{02/}}

В работе изучается дружественность деревьев. Это понятие впервые появилось в решении гипотезы Ландо о пересечении двумерных сфер в трёхмерном пространстве, подробнее см. пункт \ref{lando-intro}. Сведения пункта \ref{lando-intro} используются только для доказательства утверждения \ref{symm} и необязательны для понимания работы. 

\begin{definition}[Дружественность простому пути]
    {\it Дерево дружественно простому пути}, если рёбра этого дерева можно пронумеровать так, что для любых $k \in \mathbb N, s \in \mathbb Z$ путь\footnote{Путём между рёбрами будем называть путь, который соединяет некоторые концы этих рёбер, но не содержит сами эти рёбра. Для любой пары различных рёбер дерева такой путь существует и единственнен.} между рёбрами $k$ и $k + 1$ содержит либо оба ребра $k + 2s, k + 2s + 1$, либо ни одного из этих рёбер\footnote{В этом определении запись «ребро $k$» подразумевает «ребро под номером $k$».}. Соответствующую нумерацию также будем называть {\it дружественной}.
\end{definition}

\begin{figure}
    \center{\includegraphics[height=3cm]{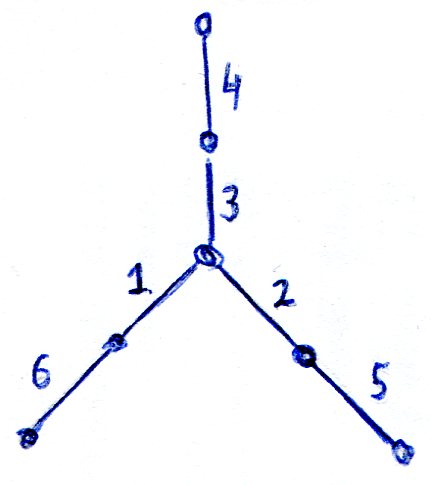}}
    \caption{\label{fig:friendly_enum}Пример дружественной нумерации.}
\end{figure}

На рисунке \ref{fig:friendly_enum} приведён пример дружественной нумерации.

\smallskip

В статье Аввакумова \cite{Avvakumov} был поднят следующий вопрос.

\begin{question}
    \label{q:path}
    Существует ли дерево, не являющееся дружественным простому пути?
\end{question}

Найти ответ на этот вопрос не удалось, однако получены интересные результаты.

\begin{figure}[h]
\begin{minipage}{0.4\linewidth}
\center{\includegraphics[height=3cm]{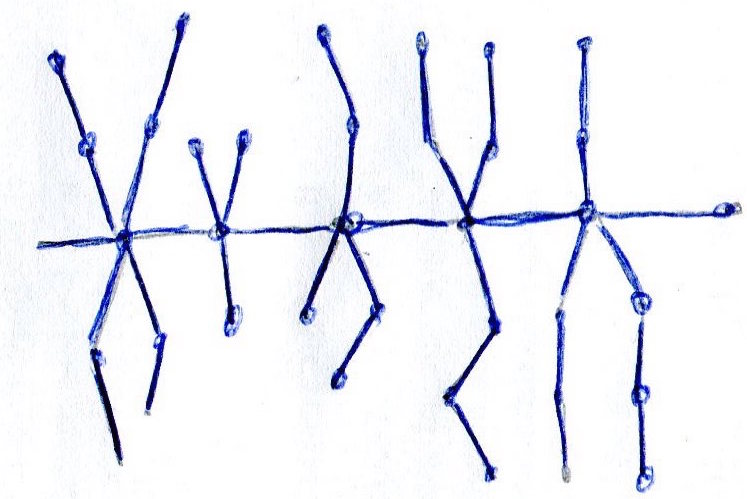}}
\caption{Пример дерева, дружественного простому пути согласно теореме \ref{low_branchiness}.\label{fig:low-branchiness-sample}}

\end{minipage}
\hfill
\begin{minipage}{0.4\linewidth}
\center{\includegraphics[height=3cm]{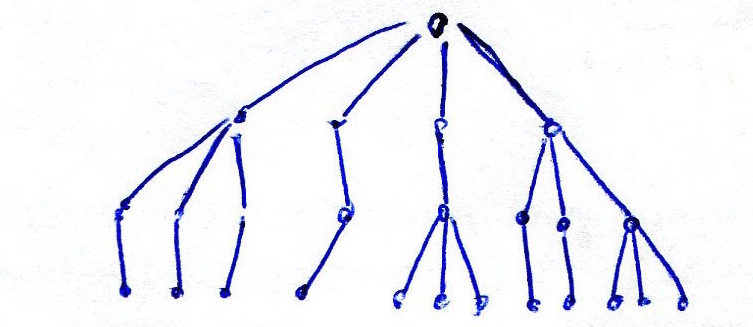}}
\caption{Пример дерева, дружественного простому пути согласно теореме \ref{parity_and_center}.\label{fig:parity-sample}}
\end{minipage}
\end{figure}

\begin{theorem}
    \label{low_branchiness}
    Пусть в дереве существует путь, содержащий все вершины этого дерева степени $3$ и более. Тогда это дерево дружественно простому пути.
\end{theorem}

Пример дерева, дружественного простому пути согласно теореме \ref{low_branchiness}, приведён на рисунке \ref{fig:low-branchiness-sample}.

\begin{theorem}
    \label{parity_and_center}
    Пусть в дереве существует вершина, расстояния от которой до всех висячих вершин одинаковы. Пусть также степени всех вершин, не являющихся висячими, чётны. Тогда это дерево дружественно простому пути.
\end{theorem}

Пример дерева, дружественного простому пути согласно теореме \ref{parity_and_center}, приведён на рисунке \ref{fig:parity-sample}.

Сформулируем гипотезы, которые, по-видимому, можно аналогично теореме \ref{parity_and_center}.

\begin{hypothesis}
    \label{g::d4}
    Любое дерево диаметра не более 4-х дружественно простому пути. 
\end{hypothesis}
\begin{hypothesis}
    \label{g::odd}
    Любое дерево, в котором степени всех вершин нечётны и существует вершина, расстояния от которой до всех висячих вершин одинаковы, дружественно простому пути. 
\end{hypothesis}

Сформулируем открытый вопрос \ref{q:all}, с которым связан ещё один результат работы. Для формулировки этого результата понадобится несколько определений, см. рис. \ref{fig:simple-def}.

\begin{figure}[h]

\center{\begin{minipage}{0.85\linewidth}
\caption{Иллюстрация к определениям «не цепляется за», «не зацеплены», «кограница». \label{fig:simple-def}}
\end{minipage}}

\vspace{3ex}
\begin{minipage}{0.60\linewidth}

\center{\includegraphics[height=2cm]{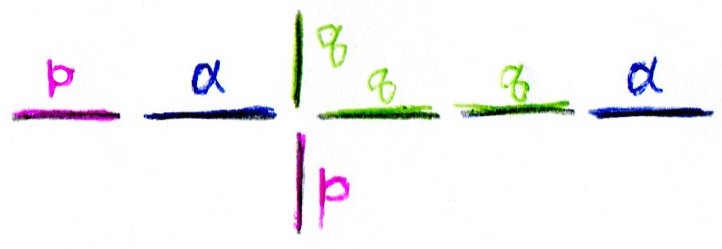}}

Множества $p$ и $q$ не зацеплены.

Множества $a$ и $q$ не зацеплены.

Множество $p$ цепляется за множество $a$.

Множество $a$ не цепляется за множество $p$.
\end{minipage}
\hfill.
\begin{minipage}{0.35\linewidth}
    \center{\includegraphics[height=2cm]{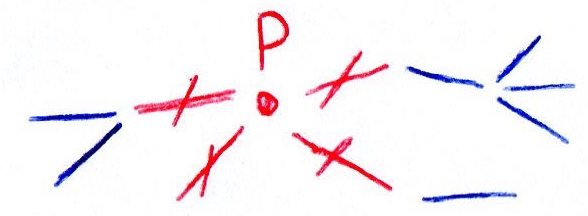}}
    
    Кограница вершины $P$.
\end{minipage}
\end{figure}

Пусть  $p$ и $q$ --- два подмножества множества рёбер некоторого дерева.

\begin{definition}[Не цепляется за]
	Множество $p$ {\it не цепляется за} множество $q$, если $p\cap q=\varnothing$, и для любых двух рёбер из $p$ соединяющий их путь в дереве содержит чётное число рёбер из $q$.
\end{definition}

\begin{definition}[Не зацеплены]
	Множества $p$ и $q$ {\it не зацеплены}, если $p$ не цепляется за $q$ и $q$ не цепляется за $p$.
\end{definition}

\begin{definition}[Кограница]
	Кограница $\delta P$ вершины $P$ --- это множество рёбер, выходящих из неё. 
\end{definition}

\begin{definition}[Дружественность в общем смысле]
    Дерево $G_1$ называется {\it дружественным} дереву $G_2$, если существует биекция между множествами их рёбер, при которой для любых двух вершин $P,Q$ графа $G_1$, соединённых путём из чётного числа рёбер, наборы рёбер в $G_2$, соответствующие $\delta P$ и $\delta Q$, не зацеплены в $G_2$. Соответствующую биекцию также будем называть {\it дружественной}.
\end{definition}

\begin{figure}
   \center{\includegraphics[height=2cm]{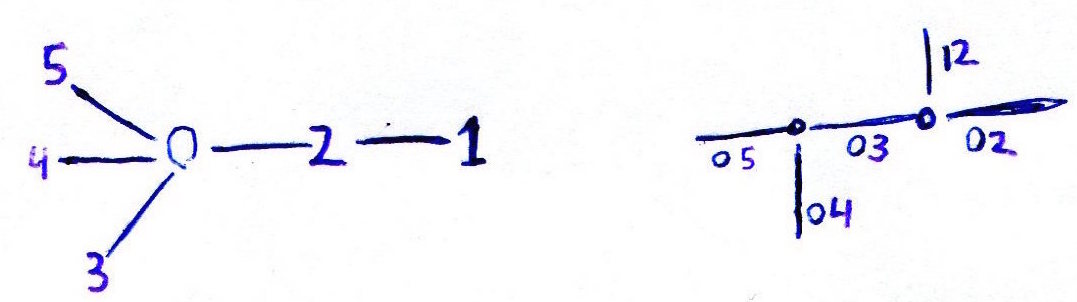}}
   \caption{\label{fig:friends}Пример дружественной биекции.}
   \center{\includegraphics[height=2cm]{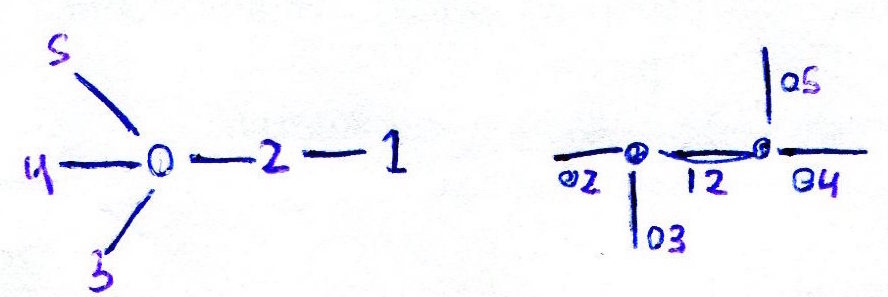}}
   \caption{\label{fig:unfriends}Пример не дружественной биекции.}
\end{figure}

На рисунках \ref{fig:friends}, \ref{fig:unfriends} приведены примеры дружественной биекции и биекции, не являющейся дружественной. То, что дерево $G$ дружественно простому пути, равносильно тому, что простой путь дружественен дереву $G$ в общем смысле.

\begin{statement}
    \label{symm}
    Дружественность симметрична. Более того, биекция, обратная к дружественной биекции, является дружественной.
\end{statement}

Утверждение \ref{symm} следует из теоремы~\ref{a_theorem} (Аввакумова). Однако прямое комбинаторное доказательство неизвестно.

Очевидно, что два дерева могут быть дружественными только когда у них одинаковое количество рёбер.

\begin{question}
\label{q:all}
Какие деревья дружественны любому дереву с таким же количеством рёбер?
\end{question}

В этой работе дружественность любому дереву с таким же количеством рёбер доказана для определённого класса деревьев. Опишем этот класс.

\begin{definition}[$n$-звезда]
    Дерево $(V, E)$ называется {\it $n$-звездой}, если
	$V = \left\{ 0, 1, \ldots, n \right\}$  
	и 
	$E = \left\{(0, 1), (0,2), \ldots, (0, n)\right\}$.
{\sloppy

}
\end{definition}

Известно, см. \cite{tasks}, что $n$-звезда дружественна любому дереву, в котором $n$ рёбер.

\begin{notation}[$CB(n_1,n_2)$\footnote{CB
                от английского contact binary (тесная двойная система) --- астрономический термин для двойных звёзд, которые могут обмениваться массой.}]
	Обозначим через $CB(n_1, n_2)$ дерево, представляющее собой объединение по одному общему ребру $n_1$-звезды и $n_2$-звезды.
\end{notation}

Обозначение «$CB(n_1, n_2)$» проиллюстрировано на рисунке \ref{fig:CB}. Отметим, что $CB(n, 1)$ --- это $n$-звезда.

\begin{figure}[h]
\begin{minipage}{0.32\linewidth}
    \center{\includegraphics[height=2cm]{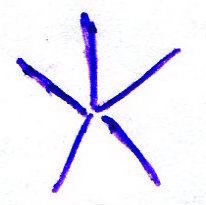}}

    $5$-звезда
\end{minipage}
\hfill
\begin{minipage}{0.32\linewidth}
    \center{\includegraphics[height=2cm]{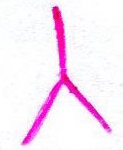}}

    $3$-звезда
\end{minipage}
\hfill
\begin{minipage}{0.32\linewidth}
    \center{\includegraphics[height=2cm]{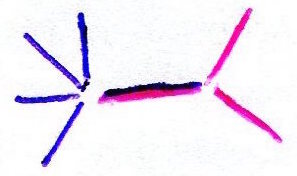}}
    $CB(5,3)$
\end{minipage}
\caption{\label{fig:CB}
Иллюстрация к обозначениям «$n$-звезда», «$CB(n_1, n_2)$».}
\end{figure}

В работе доказано следующее утверждение.

\begin{statement}
    \label{CB_n_4}
    Каждое из деревьев $CB(n, 2)$, $CB(n, 3)$, $CB(n, 4)$ дружественно любому дереву с тем же количеством рёбер.
\end{statement}

\begin{note}
    Существует дерево, не дружественное $CB(5,5)$, количество рёбер которого совпадает с количеством рёбер дерева $CB(5,5)$, см. рис. \ref{unfriendly}.
\end{note}

Утверждение \ref{CB_n_4} доказывается с помощью следующей теоремы.

\begin{theorem}[Критерий дружественности дереву диаметра 3]
	\label{3_criteria}
	Дерево $CB(n_1,n_2)$ дружественно дереву $G$ тогда и только тогда, когда, во-первых, $|E(G)| = n_1 + n_2 - 1$, и, во-вторых, в $G$ существуют два поддерева $H_1$ и $H_2$, которые пересекаются ровно по одному ребру, такие что $|E(H_1)| = n_1$, $|E(H_2)| = n_2$.
\end{theorem}

 Теорема \ref{3_criteria} несложная. Её доказательство для частного случая впервые было использовано в статье \cite{Avvakumov} как часть доказательства контрпримера к гипотезе Ландо\footnote{Подробнее о гипотезе Ландо в пункте \ref{lando-intro}.}, см. также \cite{B2}. Сама теорема \ref{3_criteria} впервые была сформулирована Матвеем Осиповым, но доказательство так и не было опубликовано, поэтому оно включено в данную работу.
        \subsection{Дружественность деревьев и гипотеза Ландо}
            \graphicspath{{03/}}

\label{lando-intro}
Понятие {\it дружественности} появилось при обобщении следующей гипотезы.

\begin{hypothesis}[Ландо]
    \label{lando}
    Пусть $M$ и $N$ --- два равномощных набора непересекающихся окружностей на сфере. Тогда существует пара кусочно-линейных сфер в трёхмёрном пространстве, которые трансверсально пересекаются по конечному набору непересекающихся окружностей, расположенных в одной сфере как $M$, а в другой --- как $N$.
\end{hypothesis}

\begin{figure}[h]
	\begin{minipage}{0.55\linewidth}
    \center{\includegraphics[height=5cm]{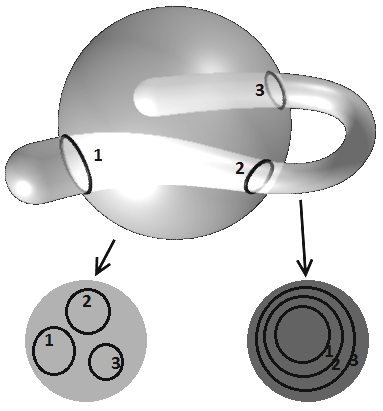}}
    \end{minipage}
    \hfill
    \begin{minipage}{0.4\linewidth}
\caption{\label{fig:lando}Наборы окружностей на сфере и соответствующее им пересечение кусочно-линейных сфер, иллюстрация из \cite{tasks}.}
	\end{minipage}
\end{figure}

Наборы окружностей на сфере и соответствующее им пересечение кусочно-линейных сфер изображены на рисунке \ref{fig:lando}. Дадим формальные определения терминам «трансверсальное пересечение» и «расположен как».

\begin{definition}[Трансверсальность пересечения, см. \cite{tasks}]
Обозначим через $B(x,\rho)\subset \mathbb{R}^3$ шар радиуса $\rho$ с центром в $x$.
Пересечение двух кусочно-линейных сфер $S,T\subset \mathbb{R}^3$ {\itshape трансверсально}, если для любой точки $x\in S\cap T$ существует $\rho>0$
такое, что и $B(x,\rho)\setminus S$, и $B(x,\rho)\cap (T\setminus S)$ состоят из двух связных компонент, а каждая компонента множества $B(x,\rho)\setminus S$ содержит компоненту множества
$B(x,\rho)\cap (T\setminus S)$.
\end{definition}

\begin{definition}[Расположен как]
Предположим, что $M$ и $N$ --- наборы из одинакового числа окружностей на кусочно-линейных сферах $S$ и $T$.
Тогда $M$ {\itshape расположен в $S$ как $N$ в $T$}, если есть биекция между связными компонентами дополнений $S\setminus M$ и $T\setminus N$,
при которой две связные компоненты дополнения $S\setminus M$ соседние тогда и только тогда, когда две соответствующие связные компоненты дополнения $T\setminus N$ соседние.
\end{definition}

Взаимное расположение окружностей на сфере можно описать с помощью графов.

\begin{definition}[Двойственный граф]
Пусть $M$ --- объединение непересекающихся окружностей в кусочно-линейной сфере $S$. Определим {\it двойственный к $M$} граф $G=G(S,M)$ следующим образом. Вершины --- связные компоненты дополнения $S\setminus M$. Вершины соединяются ребром, если соответствующие компоненты --- соседи.
\end{definition}

\begin{figure}[h]
	\begin{minipage}{0.55\linewidth}
    \center{\includegraphics[height=5cm]{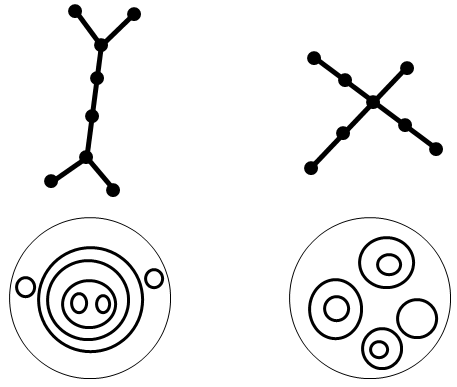}}
    \end{minipage}
    \hfill
    \begin{minipage}{0.4\linewidth}
\caption{\label{fig:sp-tr}Соответствие между наборами окружностей на сфере и деревьями, иллюстрация из \cite{tasks}.}
\end{minipage}
\end{figure}

Можно показать, что граф, двойственный к набору непересекающихся окружностей на сфере, является деревом \cite{Avvakumov}. Примеры двойственных графов изображены на рисунке \ref{fig:sp-tr}. Критерий существования кусочно-линейных сфер, реализующих то или иное пересечение, можно описать на языке деревьев.

\begin{theorem}[Аввакумова]
\label{a_theorem}
Пусть $M$ и $N$ --- два равномощных набора непересекающихся окружностей на сфере. Тогда пара кусочно-линейных сфер в трёхмёрном пространстве, которые трансверсально пересекаются по конечному набору непересекающихся окружностей, расположенных в одной сфере как $M$, а в другой --- как $N$, существует тогда и только тогда, когда деревья, двойственные к наборам $M$ и $N$, дружественны.
\end{theorem}

Доказательство этой теоремы приведено в \cite{Avvakumov}. С помощью неё Аввакумовым был построен первый контрпример к гипотезе Ландо, см. рис. \ref{unfriendly}. Позже Белоусовым~\cite{Belousov} был найден другой, меньший (и, по-видимому, минимальный), контрпример, см. рис. \ref{fig:b}. Cм.~также элементарное изложение основ темы в \cite{tasks}.

\begin{figure}[h]
\begin{minipage}{0.55\linewidth}
    \center{\includegraphics[height=5cm]{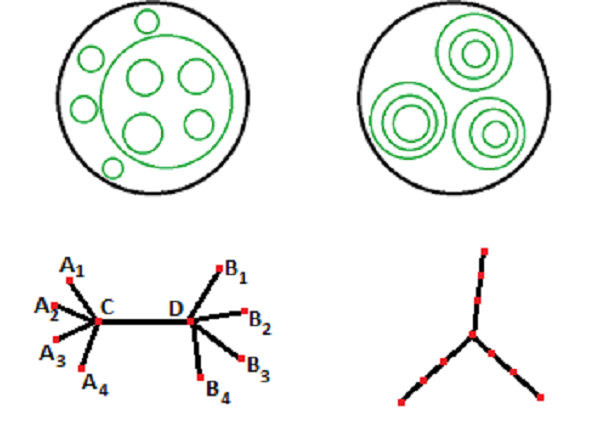}}
    \end{minipage}
    \hfill
    \begin{minipage}{0.4\linewidth}
\caption{Первый контрпример к гипотезе Ландо\label{unfriendly}, иллюстрация из \cite{tasks}.}
\end{minipage}
\end{figure}

\begin{figure}[h]
\begin{minipage}{0.7\linewidth}
    \center{\includegraphics[height=5cm]{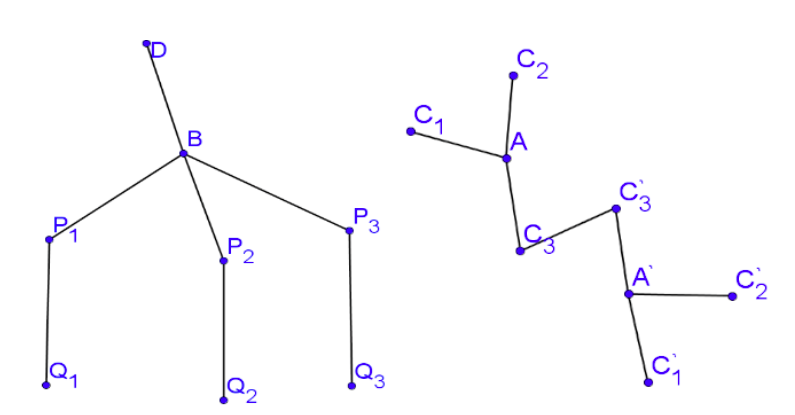}}
    \end{minipage}
    \hfill
    \begin{minipage}{0.2\linewidth}
\caption{\label{fig:b}Контрпример Белоусова, иллюстрация из \cite{Belousov}.}
\end{minipage}
\end{figure}
                
    \section{Доказательства}
        \subsection{Доказательство теоремы \ref{low_branchiness}}
            \graphicspath{{05/}}

Пусть дерево $G$ удовлетворяет условию теоремы \ref{low_branchiness}. Введём обозначения для частей этого дерева.

\begin{definition}[Ствол]
{\it Ствол} дерева $G$ --- это путь в $G$, который содержит все вершины дерева $G$ степени 3 и более и конец которого является висячей вершиной дерева $G$.
\end{definition}

Поясним, что {\it первой вершиной ствола} будем называть первую вершину соответствующего пути, {\it последней вершиной ствола} --- последнюю. Для полноты приведём доказательство следующей простой леммы.

\begin{lemma}
В дереве $G$, удовлетворяющем условию теоремы \ref{low_branchiness}, можно выбрать ствол.
\end{lemma}

\begin{proof}
По условию теоремы в дереве $G$ существует путь $p$, который содержит все вершины степени $3$ и более. Пусть его конец $Q$ не является висячей вершиной. Тогда у $Q$ степень 2 или более, и существует ребро $e$, выходящее из $Q$ и не принадлежащее пути $p$. Продолжим $p$ по ребру $e$. Будем продолжать $p$, пока не дойдём до висячей вершины. Получим путь $p_1$, который содержит все вершины дерева $G$ степени 3 и более (поскольку содержит $p$) и чей конец является висячей вершиной.
\end{proof}

\begin{figure}[h]
\begin{minipage}{0.49\linewidth}
\center{\includegraphics[height=4cm]{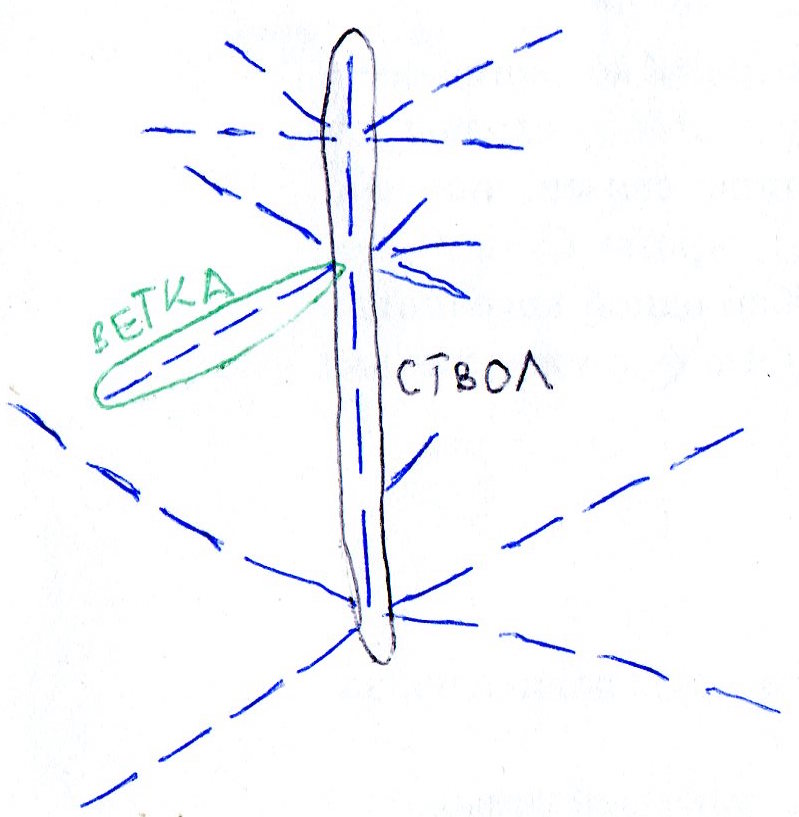}}
\end{minipage}
\hfill
\begin{minipage}{0.49\linewidth}
\center{\includegraphics[height=4cm]{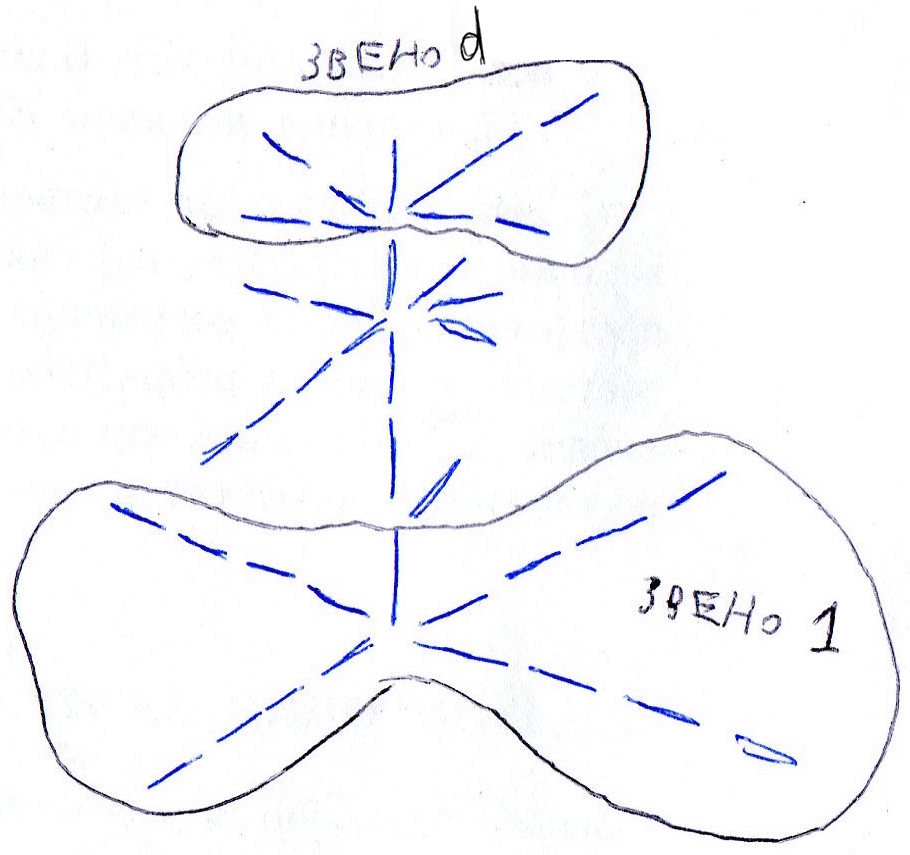}}
\end{minipage}
\caption{\label{fig:branch}Иллюстрация к определениям «ствол», «ветка», «звено».}
\end{figure}

Выберем в дереве $G$ ствол и зафиксируем его. В дальнейшем под словом «ствол» мы будем подразумевать именно этот зафиксированный ствол.

\begin{definition}[Ветка]
{\it Ветка} дерева $G$ --- это путь в дереве $G$
\begin{itemize}
\item первая вершина которого лежит на стволе,
\item последняя вершина которого является висячей,
\item рёбра которого не лежат на стволе.
\end{itemize}
Ветки, в которых чётное количество рёбер, будем называть {\it чётными}, а ветки, в которых нечётное --- {\it нечётными}. Ближайшее к стволу ребро ветки будем называть {\it первым}, соседнее с ним --- {\it вторым}, и так далее. Висячее ребро ветки будем называть {\it последним} ребром ветки.
\end{definition}

\begin{definition}[Звено]
{\it $m$-м звеном дерева $G$} будем называть подграф дерева $G$, содержащий
\begin{itemize}
\item $m$-ю вершину ствола, то есть вершину ствола, лежащую на расстоянии $m - 1$ от первой вершины ствола,
\item все ветки, выходящие из $m$-й вершины,
\item ребро ствола, выходящее из $m$-й вершины и лежащее на пути из неё в последнюю вершину ствола.
\end{itemize}
\end{definition}

Обозначим через $d$ количество рёбер в стволе дерева $G$. Тогда каждое ребро дерева $G$ лежит ровно в одном из звеньев $1$, $2$, $\dots$, $d$. \footnote{Именно для этого нам понадобилось, чтобы последняя вершина ствола была висячей. Иначе рёбра веток, начинающихся в последней вершине ствола, не лежат ни в одном из звеньев.}

\smallskip

{\it Описание дружественной нумерации рёбер дерева $G$}.

Обозначим через $n_1$, $n_2$, $\dots$, $n_d$ количества рёбер в звеньях $1$, $2$, $\dots$, $d$, соответственно. Далее рёбра первого звена получат номера от $1$ до $n_1$, рёбра второго звена получат номера от $n_1 + 1$ до $n_1 + n_2$, $\dots$, рёбра последнего звена получат номера от $|E(G)| - n_d + 1$ до $|E(G)|$.

Опишем нумерацию рёбер внутри звена, см. рис. \ref{fig:enum-algo}, рис. \ref{fig:enum-draw}.

\begin{algorithm}
\caption{\label{fig:enum-algo}Алгоритм нумерации рёбер внутри звена.}
\vspace{3ex}
\begin{minipage}{0.40\linewidth}
\For{$i \leftarrow  1, 2, \ldots, t_1$}{
	
	\For{$j \leftarrow  1, 2, \ldots, L_{1,i}$}{
	
		$odd[i][j] \leftarrow k$ 
	
		$k \leftarrow k + 1$ 
	}
}

$e_{trunk} \leftarrow k$

$k \leftarrow k + 1$

\For{ $i \leftarrow  1, 2, \ldots, t_2$ }{

		$even[i][1] \leftarrow k$

		$k \leftarrow k + 1$
	}

\For{ $i \leftarrow t_2, t_2 - 1, \ldots, 1$ }{
	
	\For{$j \leftarrow  2, 3, \ldots, L_{2,i}$}{
	
		$even[i][j] \leftarrow k$
	
		$k \leftarrow k + 1$
	}
}
\end{minipage}
\hfill
\begin{minipage}{0.49\linewidth}
{\bf Обозначения}. Здесь «$k$» означает текущий номер, стрелка означает операцию присвоения. Нечётные ветки обозначены через $odd[1]$, $odd[2]$, $\ldots$, $odd[t_1]$, их длины --- через $L_{1,1}$, $L_{1,2}$, $\ldots$, $L_{1,t_1}$. Рёбра $i$-ой нечётной ветки обозначены через $odd[i][1]$, $odd[i][2]$, $\ldots$, $odd[i][L_{1, i}]$. Обозначения чётных веток $even[1]$, $even[2]$, $\ldots$, $even[t_2]$, их рёбер и длин $L_{2,1}$, $L_{2,2}$, $\ldots$, $L_{2,t_2}$ аналогичны. Ребро звена, лежащее на стволе, обозначено через $e_{trunk}$.
\end{minipage}
\end{algorithm}

\begin{figure}[h]
\center{
	\begin{minipage}{0.33\linewidth}
	\center{\includegraphics[height=4cm]{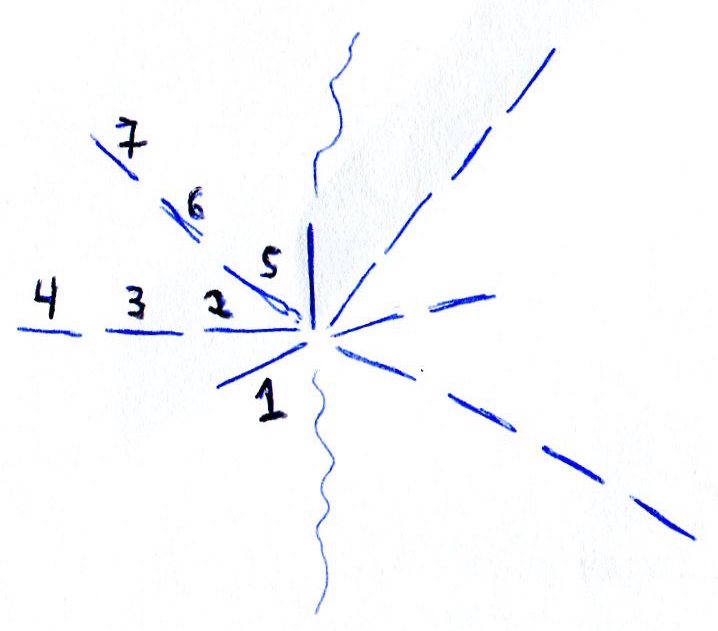}}
	\center{\includegraphics[height=4cm]{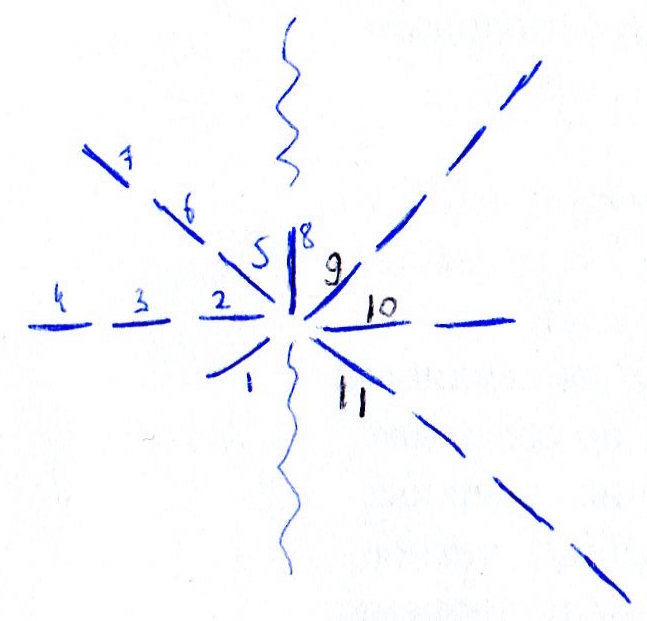}}
	\end{minipage}
	\begin{minipage}{0.33\linewidth}
	\center{\includegraphics[height=4cm]{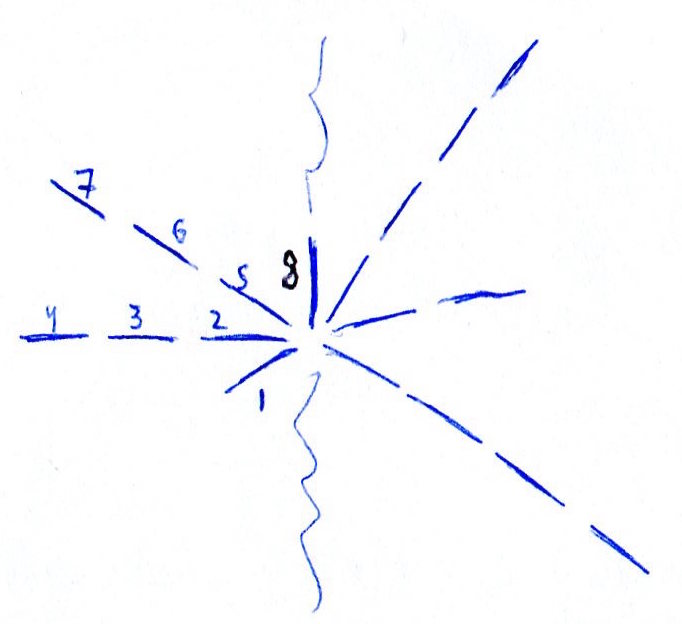}}
	\center{\includegraphics[height=4cm]{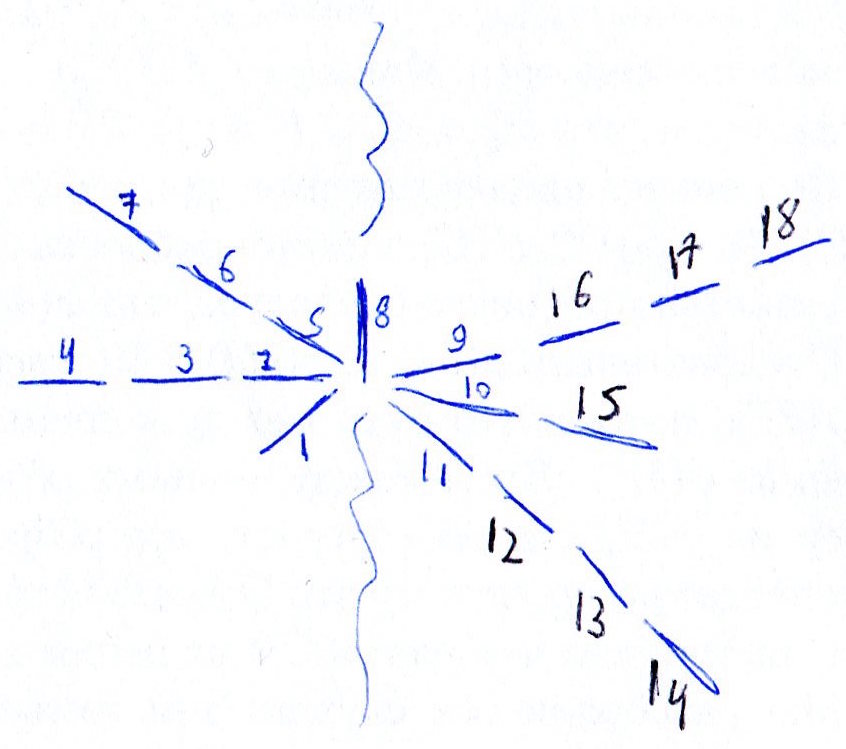}}
	\end{minipage}
}
\caption{\label{fig:enum-draw}Пример построения нумерации звена.}
\end{figure}

Обозначим через $t_1$ количество нечётных веток в узле, через $t_2$ --- количество чётных. Пронумеруем нечётные ветки числами $1$, $2$, $\dots$, $t_1$, чётные --- числами $1$, $2$, $\dots$, $t_2$.

Если в звене есть нечётные ветки, то сначала присваиваем номера их рёбрам. По очереди проходим по всем нечётным веткам и в каждой присваиваем номера от ствола к висячему ребру подряд: сначала рёбрам с первого по последнее первой нечётной ветки, потом ребрам с первого по последнее второй нечётной ветки и так далее.

Следующий номер присваиваем ребру звена, лежащему на стволе.

Если в звене есть чётные ветки, то переходим к ним. Нумерация рёбер чётных веток двухэтапна. Сначала по очереди проходим по всем чётным веткам $1$, $2$, $\dots$, $t_2$ и в каждой присваиваем очередной номер ребру, ближайшему к стволу. Потом проходим по чётным веткам в обратном порядке и в каждой ветке нумеруем оставшиеся рёбра от ствола к висячему ребру. Сначала рёбра со второго по последнее $t_2$-й чётной ветки. Потом рёбра со второго по последнее предпоследней чётной ветки и так далее. Последними в звене нумеруем рёбра со второго по последнее первой чётной ветки.

\smallskip

Далее в доказательстве теоремы \ref{low_branchiness}, там, где это уместно, мы будем использовать запись «ребро $k$» как сокращение для слов «ребро, получившее номер $k$ в построенной нами нумерации». Для доказательства теоремы \ref{low_branchiness} нам понадобится следующее определение.

\begin{definition}[Самостоятельность]
Будем говорить, что номер $k$ {\it самостоятелен}, если в построенной нами нумерации для произвольного целого $s$ путь между $k$ и $k + 1$ в дереве $G$ содержит либо оба ребра $k + 2s, k + 2s + 1$, либо ни одного из этих рёбер.
\end{definition}

Очевидно, что дружественность дерева $G$ простому пути эквивалентна самостоятельности всех номеров. А самостоятельность номера $k$ эквивалентна тому, что рёбра на пути между $k$ и $k + 1$ можно разбить на пары, такие что каждая состоит из рёбер $k + 2s$, $k + 2s + 1$ для некоторого целого $s$.

Для номеров $k$, таких, что рёбра $k$ и $k + 1$ имеют общую вершину, самостоятельность очевидна. 

Пусть теперь рёбра $k$ и $k + 1$ не имеют общей вершины. По построению такое возможно только когда $k$ --- висячее ребро ветки. Обозначим через $i$ номер этой ветки, через $L$ --- её длину, через $m$ --- номер узла, которому она принадлежит.

\begin{figure}
\center{
\begin{minipage}{0.8\linewidth}
\begin{minipage}{0.45\linewidth}
\center{\includegraphics[height=5cm]{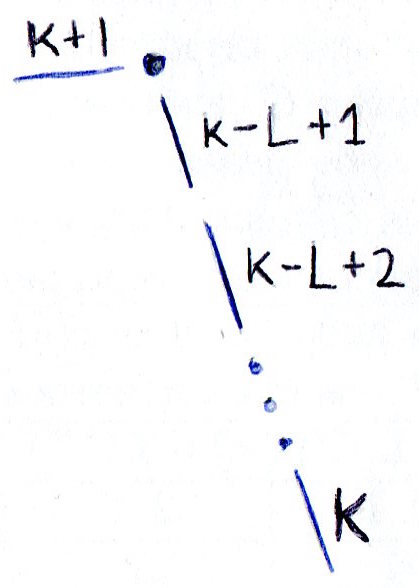}}
\caption{\label{fig:05odd}$k$ --- висячее ребро нечётной ветки.}
\end{minipage}
\hfill
\begin{minipage}{0.45\linewidth}
\center{\includegraphics[height=5cm]{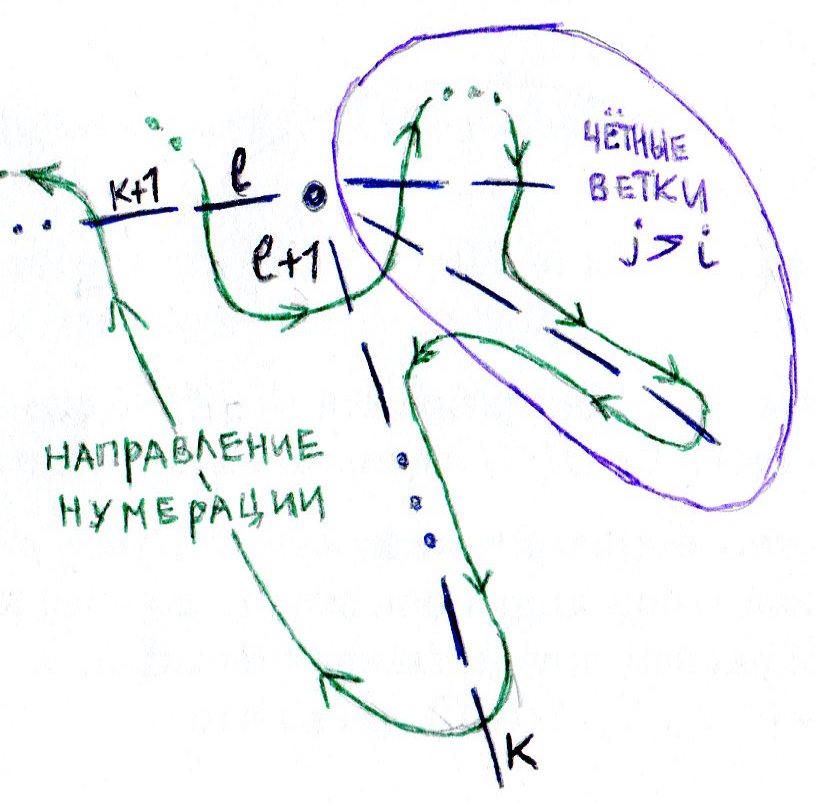}}
\caption{\label{fig:05even}$k$ --- висячее ребро чётной ветки.}
\end{minipage}
\end{minipage}
}
\end{figure}

{\it Случай: $L$ нечётно,} см. рис. \ref{fig:05odd}. По построению $k + 1$ --- либо первое ребро другой нечётной ветки $m$-го звена, либо ребро ствола $m$-го звена. В обоих случаях ребро $k + 1$ смежно с $m$-й вершиной ствола. Поэтому между рёбрами $k$ и $k + 1$ лежат рёбра $i$-й нечётной ветки и только они. Мы нумеровали рёбра внутри нечётных веток подряд, поэтому рёбра $i$-й нечётной ветки --- это $k - L + 1$, $k - L + 2$, $\ldots$, $k - 1$, $k$. Рассматриваемая нами ветка нечётная, поэтому между рёбрами $k$ и $k + 1$ лежит чётное количество рёбер. Кроме этого, номера $k - L + 1$ и $k$ имеют одинаковую чётность. Разобьём рёбра между рёбрами $k$ и $k + 1$ на пары так:

$$(k - L + 1, k - L + 2), \ldots, (k - 2, k - 1).$$

Каждая из этих пар состоит из рёбер $k + 2s$, $k + 2s + 1$ для некоторого целого $s$. А значит, номер $k$ самостоятелен.

\smallskip

{\it Случай: $L$ --- чётно,} см. рис. \ref{fig:05even}. Если $i = 1$, то $k + 1$ --- это ребро, не принадлежащее $m$-му звену и смежное с вершиной ствола, имеющей номер $m + 1$. В остальных случаях $k + 1$ --- это ближайшее к стволу ребро чётной ветки c номером $i - 1$. Таким образом, между ребром $k + 1$ и $m$-й вершиной ствола лежит единственное ребро. Обозначим его через $l$. Получается, что между рёбрами $k$ и $k + 1$ лежат рёбра $i$-й чётной ветки и ребро $l$. По построению, ближайшее к стволу ребро $i$-й чётной ветки имеет номер $l + 1$, а остальные рёбра этой ветки --- это $k - L + 2$, $k - L + 3$, $\ldots$, $k - 1$, $k$. Разобьём рёбра между рёбрами $k$ и $k + 1$ на пары так:

$$(l, l + 1),
	(k - L + 2, k - L + 3)
	, \ldots, 
	(k - 2, k - 1).$$

То, что каждая из остальных пар $(k - L + 2, k - L + 3), \ldots, (k - 2, k - 1)$ состоит из рёбер $k + 2s$, $k + 2s + 1$ для некоторого целого $s$, следует из чётности $L$.

Докажем, что $k - l$ чётно. Число $k - l$ равно количеству рёбер $l + 1$, $l + 2$, $\ldots$, $k - 1$, $k$. Оно чётно, так как по построению рёбра $l + 1$, $l + 2$, $\ldots$, $k - 1$, $k$ --- это все рёбра $i$-й чётной ветки, а также все рёбра чётных веток, получивших номера большие, чем $i$.

Таким образом, номер $k$ самостоятелен и в этом случае.

\qed

        \subsection{Доказательство теоремы \ref{parity_and_center}}
            \graphicspath{{06/}}
Отметим, что Аввакумовым предложено более простое, не использующее индукцию доказательство теоремы \ref{parity_and_center}. Однако в работе предлагается индуктивное доказательство, поскольку кажется, что его проще обобщить для доказательства гипотез \ref{g::d4} и \ref{g::odd}.

Для доказательства теоремы \ref{parity_and_center} мы докажем более сильное утверждение \ref{forced_parity_and_center}. Для того, чтобы его сформулировать, введём следующие обозначения.

\begin{notation}[Висячие рёбра]
    Для дерева $G$ будем обозначать через $H(G)$ множество его висячих рёбер.
\end{notation}

\begin{notation}[Знак сравнения]
  Пусть $p, q \subset \mathbb Z$. Будем говорить, что $p < q$, если $\forall e \in p, \forall w \in q$ выполнено $e < w$, см. рис. \ref{fig:less}.

  Пусть $e \in \mathbb Z, q \subset \mathbb Z$. Будем говорить, что $e < q$, если $\forall w \in q$ выполнено $e < w$.
\end{notation}

\begin{figure}[h]
  \center{\includegraphics[height=1.5cm]{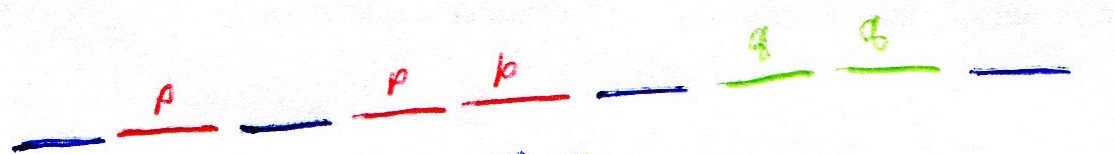}}
  \caption{\label{fig:less}$p < q$}
\end{figure}

\begin{definition}
    Будем говорить, что биекция $\varphi: E(G) \to \{1, 2, \ldots |E(G)|\}$ из множества рёбер дерева $G$ во множество рёбер простого пути является {\it листокрайной}, если выполнено свойство
        $$\varphi(E(G)\setminus H(G)) < \varphi(H(G)).$$
\end{definition}

\begin{statement}
  \label{forced_parity_and_center}
  Пусть $G$ --- дерево, удовлетворяющее условию теоремы \ref{parity_and_center}. Тогда существует листокрайная дружественная биекция
  $\varphi: E(G) \to \{1, 2, \ldots, |E(G)|\}$ .
\end{statement}

Ввиду симметричности дружественности (см. утв. \ref{symm}) из утверждения \ref{forced_parity_and_center} следует теорема \ref{parity_and_center}. Более того, биекция, обратная к построенной в доказательстве утверждения \ref{forced_parity_and_center}, задаёт дружественную нумерацию рёбер дерева $G$. Для доказательства утверждения \ref{forced_parity_and_center} введём обозначения.

\begin{notation}[Висячие вершины]
  Для дерева $G$ обозначим через $L(G)$ множество его висячих вершин.
\end{notation}

\begin{notation}[$G'$]
    Для дерева $G$ обозначим через $G'$ дерево, у которого $V = {V(G) \setminus L(G)}$, $E = {E(G) \setminus H(G)}$.
\end{notation}

\begin{notation}[$f_{pa}(e)$]
    В дереве $G$, если $H(G') \not = \varnothing$, то для ребра $e \in H(G)$ будем обозначать смежное с ним ребро из $H(G')$ через $f_{pa}(e)$.
\end{notation}

\begin{notation}[$f_{pa}^{-1}(e)$]
    В дереве $G$, для ребра $e \in H(G')$ будем обозначать через $f^{-1}_{pa}(e)$ можество рёбер из $H(G)$, которые имеют общую вершину с $e$.
\end{notation}

Обозначения «$G'$», «$f_{pa}(e)$» проиллюстрированы рисунком \ref{fig:fpa}.

\begin{figure}[h]
  \begin{minipage}{0.49\linewidth}
    \center{\includegraphics[height=3cm]{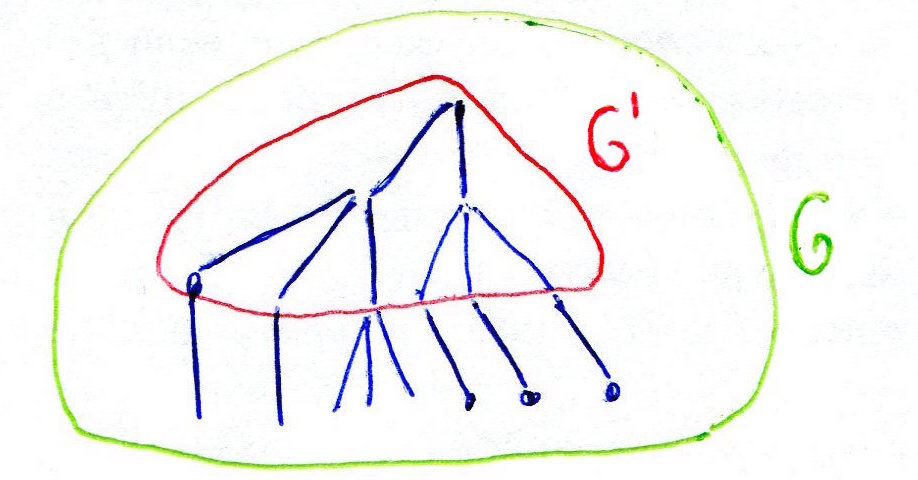}}
  \end{minipage}
  \hfill
  \begin{minipage}{0.49\linewidth}
    \center{\includegraphics[height=3cm]{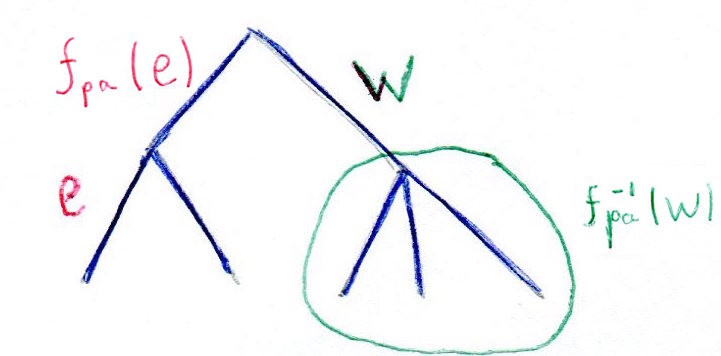}}
  \end{minipage}
  \caption{\label{fig:fpa}Иллюстрации к обозначениям «$G'$», «$f_{pa}(e)$».}
\end{figure}

\begin{proof}[Начало доказательства утверждения \ref{forced_parity_and_center}]
Обозначим через $C$ вершину дерева $G$, расстояние от которой до всех висячих вершин одинаково. Обозначим через $\rho$ расстояние от $C$ до висячих вершин дерева $G$. Будем доказывать утверждение \ref{forced_parity_and_center} индукцией по $\rho$.

База индукции, когда $\rho = 0$, тривиальна, так как в этом случае дерево $G$ состоит из одной вершины.

Пусть $\rho > 0$. Предположим по индукции, что для любого дерева $G_1$, в котором все висячие вершины находятся на расстоянии $\rho_1 < \rho$ от некоторой вершины, а все не висячие вершины имеют чётную степень, существует листокрайная дружественная биекция в простой путь. Заметим, что к дереву $G'$ применимо предположение индукции. Значит, существует листокрайная дружественная биекция $\varphi': G' \to \{1, 2, \ldots, |E(G')|\}$. Очевидно, существует её продолжение
$\varphi: G \to \{1, 2, \ldots, |E(G)|\}$,
обладающее свойством {\it противохода}:
\begin{multline*}
\text{Для любой пары рёбер } e_1, e_2 \in H(G) \\
\text{ верно } \varphi(f_{pa}(e_1)) < \varphi(f_{pa}(e_2)) \Rightarrow \varphi(e_1) > \varphi(e_2)
\end{multline*}

Пример такого продолжения изображён на рисунке \ref{fig:continue}.

\begin{figure}[h]
  \begin{minipage}{0.49\linewidth}
  \center{\includegraphics[height=3cm]{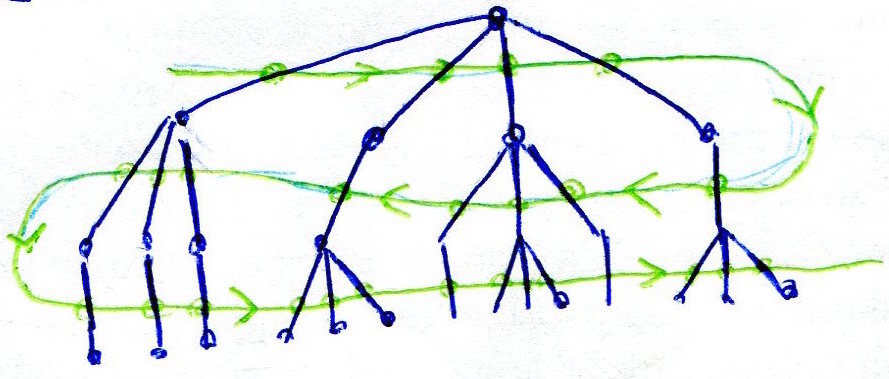}}
  \end{minipage}
  \hfill
  \begin{minipage}{0.45\linewidth}
  \caption{\label{fig:continue}Продолжение листокрайной биекции, обладающее свойством противохода.}
  \end{minipage}
\end{figure}

Очевидно, построенная биекция $\varphi$ является листокрайной. Для доказательства шага индукции осталось доказать её дружественность. Нам понадобится следующая лемма.

\begin{lemma}[О нечётности]
  \label{odd_lemma}
  Пусть $G$ --- дерево, удовлетворяющее условию теоремы \ref{parity_and_center}. Пусть $C, Q, P$ --- вершины этого дерева, такие, что $C$ находится ото всех висячих вершин на одинаковом расстоянии $\rho$, вершина $Q$ --- висячая, вершина $P$ --- смежная с висячей, возможно, другой. Тогда расстояние от $P$ до $Q$ нечётно.
\end{lemma}
\begin{figure}[h]
  \center{\includegraphics[height=5cm]{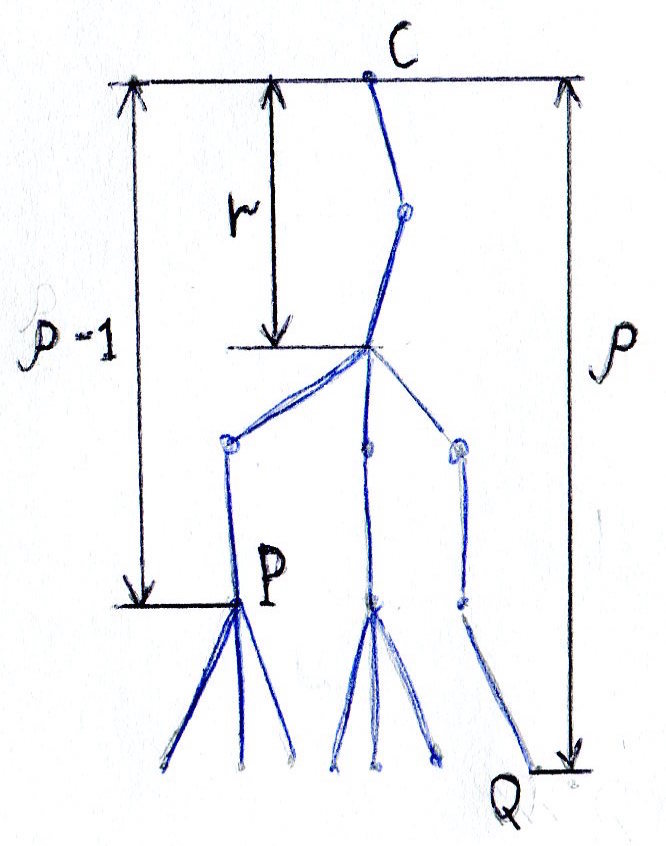}}
  \caption{\label{fig:odd_lemma}Лемма о нечётности.}
\end{figure}
\begin{proof}
  См. рис. \ref{fig:odd_lemma}. Расстояние от вершины $Q$ до $C$ равно $\rho$. Расстояние от вершины $P$ до $C$ равно $\rho - 1$. Пусть $r$ --- количество общих рёбер путей от $Q$ до $C$ и от $P$ до $C$. Тогда расстояние между $Q$ и $P$ равно $(\rho - r) + (\rho - 1 - r) = 2(\rho - r) - 1$, а значит нечётно.
\end{proof}

{\it Продолжение доказательства утверждения \ref{forced_parity_and_center}}. Имеем $V(G) = L(G) \sqcup L(G') \sqcup L(G'') \sqcup V(G''')$ (для маленьких $\rho$ некоторые из этих множеств пусты). В таблице на рис. \ref{case-table} на пересечении строки и столбца указан случай, в котором доказано, что образы кограниц вершин, лежащих в соответствующих множествах и находящихся друг от друга на ненулевом чётном расстоянии, не зацеплены.

\begin{figure}[h]
\caption{\label{case-table}Таблица покрытия случаями.}
\center{\begin{tabular}{|c|c|c|c|c|}
\hline
          & $L(G)$ & $L(G')$ & $L(G'')$ & $V(G''')$ \\
\hline
   $L(G)$ &      2 &       1 &        5 &         4 \\
\hline
  $L(G')$ &      1 &       6 &        1 &         4 \\
\hline
 $L(G'')$ &      5 &       1 &        3 &         3 \\
\hline
$V(G''')$ &      4 &       4 &        3 &         3 \\
\hline
\end{tabular}}
\end{figure}

{\it Случай 1: $P \in L(G'), Q \in L(G) \cup L(G'')$.}  Вершины из $L(G'')$ лежат на нечётном расстоянии от $P$ согласно лемме \ref{odd_lemma} о нечётности, применённой к дереву $G'$. Вершины из $L(G)$ лежат на нечётном расстоянии от $P$ согласно лемме \ref{odd_lemma} о нечётности, применённой к дереву $G$. Значит, в этом случае вершины $P$ и $Q$ не могут лежать друг от друга на чётном расстоянии.

{\it Случай 2: $P, Q \in L(G)$.} Образы $\varphi(\delta P)$ и $\varphi(\delta Q)$ кограниц вершин $P$ и $Q$, лежащих друг от друга на чётном расстоянии, не зацеплены, так как каждый содержит всего одно ребро.

{\it Случай 3: $P, Q \in V(G''') \cup L(G'')$.} Раз $P, Q \in V(G''') \cup L(G'')$, то $\delta Q \subset E(G')$ и $\delta P \subset E(G')$. Тогда из того, что $\varphi$ --- продолжение дружественной биекции $\varphi': E(G') \to \{ 1, 2, \ldots |E(G')| \}$, следует, что образы $\varphi(\delta P)$ и $\varphi(\delta Q)$ не зацеплены.

{\it Случай 4: $P \in V(G''')$, $Q \in L(G') \cup L(G)$.} Так как $Q \in L(G') \cup L(G)$, то $\delta Q \subset H(G') \cup H(G)$. Заметим, что $\delta P \subset E(G) \setminus (H(G') \cup H(G))$, а значит $\varphi(\delta P) < \varphi(\delta Q)$, из-за листокрайности биекции $\varphi$ и листокрайности биекции $\varphi'$. Следовательно образы $\varphi(\delta P)$ и $\varphi(\delta Q)$ не зацеплены.

{\it Случай 5: $P \in L(G'')$, $Q \in L(G)$.} Так как $Q \in L(G)$, то $\delta Q \subset H(G)$. Заметим, что $\delta(P) \subset E(G) \setminus H(G)$, а значит $\varphi(\delta P) < \varphi(\delta Q)$ из-за листокрайности биекции  $\varphi$ и листокрайности биекции $\varphi'$. Следовательно образы $\varphi(\delta P)$ и $\varphi(\delta Q)$ не зацеплены.

{\it Случай 6: $P, Q \in L(G')$,} см. рис. \ref{fig:case6}. Докажем, для всех пар различных $P, Q \in L(G')$, что $\varphi(\delta P)$ и $\varphi(\delta Q)$ не зацеплены.
  
Если  $H(G')$ пусто, то $\rho = 1$ и $L(G')$ состоит из одной вершины, а значит выбрать пару различных вершин из $L(G')$ невозможно. Пусть теперь $H(G')$ не пусто. Пусть ребро $e_P \in H(G')$ смежно с $P$, ребро $e_Q \in H(G')$ смежно с $Q$. Из того, что $P$ и $Q$ различны, следует, что $e_P$ и $e_Q$ --- это разные рёбра. Предположим без потери общности, что $\varphi(e_P) < \varphi(e_Q)$.

\begin{figure}[h]
  \center{\includegraphics[height=5cm]{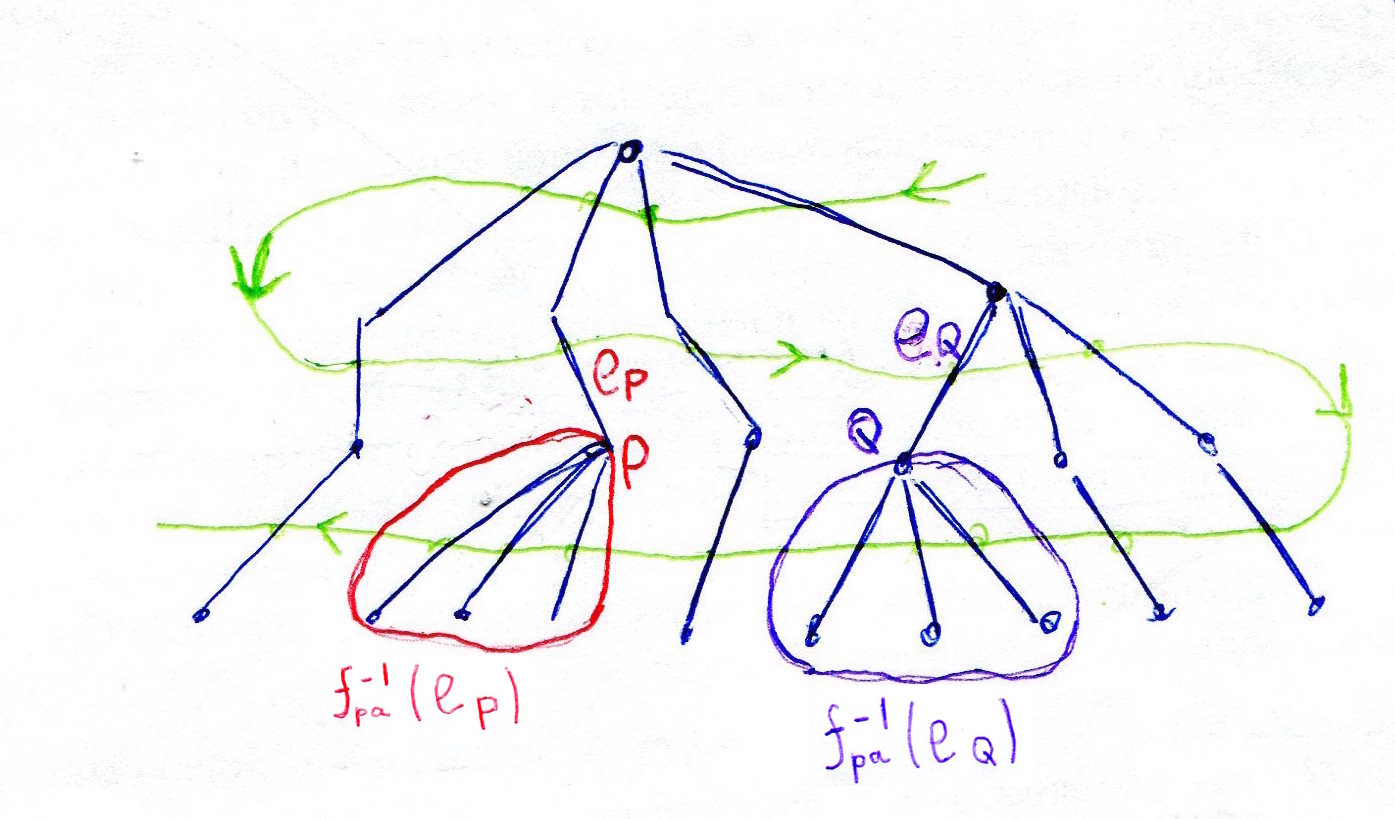}}
  \caption{\label{fig:case6}Случай 6.}
\end{figure}

Ясно, что $\delta Q = \{e_Q\} \cup f_{pa}^{-1}(e_Q)$, $\delta P = \{e_P\} \cup f_{pa}^{-1}(e_P)$.

По нашему предположению $\varphi(e_P) < \varphi(e_Q)$.
Из того, что $e_Q \in E(G)\setminus H(G)$, $f_{pa}^{-1}(e_Q) \subset H(G)$ и листокрайности биекции $\varphi$ следует что $\varphi(e_Q) < \varphi(f_{pa}^{-1}(e_Q))$.
Из свойства противохода следует, что
$\varphi(f_{pa}^{-1}(e_Q)) < \varphi(f_{pa}^{-1}(e_P))$.

Следовательно $\varphi(e_P) < \varphi(\delta Q) < \varphi(f_{pa}^{-1}(e_P))$. Пути между концами рёбер из $\varphi(\delta Q)$ не содержат рёбер из $\varphi(\delta P)$, поэтому  $\varphi(\delta Q)$ не цепляется за $\varphi(\delta P)$. Пути между концами рёбер из $\varphi(\delta P)$ либо не содержат рёбер из $\varphi(\delta Q)$, либо содержат все рёбра из $\varphi(\delta Q)$. Вершина $Q$ имеет чётную степень по условию теоремы, следовательно в $\varphi(\delta Q)$ чётное количество рёбер. Получается, что и образ $\varphi(\delta P)$ не цепляется за образ $\varphi(\delta Q)$.        

\smallskip

Мы разобрали все случаи, тем самым доказали дружественность построенного продолжения, а значит шаг индукции и утверждение~\ref{forced_parity_and_center}.
\end{proof}
        \subsection{Доказательство теоремы \ref{3_criteria}}
            \graphicspath{{07/}}
\label{sec:7}
Обозначим через $C_1$ и $C_2$ вершины $CB(n_1, n_2)$, из которых исходит $n_1$ и $n_2$ рёбер, соответственно, см. рис. \ref{fig:C1C2}.

\begin{figure}[h]
\begin{minipage}{0.49\linewidth}
\center{\includegraphics[height=3cm]{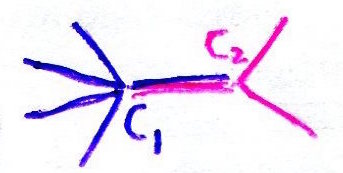}}
\caption{\label{fig:C1C2}$CB(5,3)$.}
\end{minipage}
\hfill
\begin{minipage}{0.49\linewidth}
\center{\includegraphics[height=4cm]{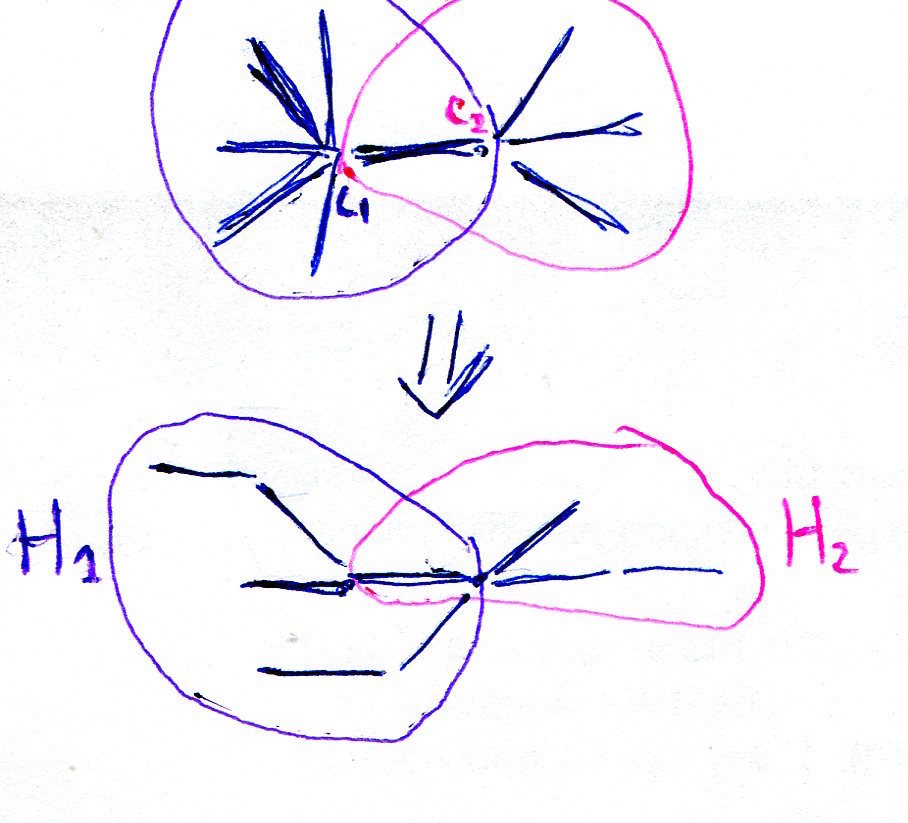}}
\caption{\label{fig:CB-bijection}Пример дружественной биекции.}
\end{minipage}
\end{figure}

\begin{proof}[Доказательство достаточности в теореме \ref{3_criteria}]
Пусть $H_1$ и $H_2$ --- поддеревья из условия теоремы. Пусть $e$ --- ребро, по которому они пересекаются. Построим дружественную биекцию $\varphi: E(CB(n_1, n_2)) \to E(G)$ так:
\begin{itemize}
\item $\varphi(C_1C_2) := e$,
\item остальные рёбра кограницы $\delta C_1$ произвольно переходят в остальные рёбра, принадлежащие $E(H_1)$,
\item остальные рёбра кограницы $\delta C_2$ произвольно переходят в остальные рёбра, принадлежащие $E(H_2)$.
\end{itemize}

\noindent
Пример построенной таким образом биекции изображён на рисунке \ref{fig:CB-bijection}.
Докажем, что $\varphi$ дружественная. Ясно, что образ кограницы каждой вершины дерева $CB(n_1, n_2)$ связен. Образы $\varphi(\delta C_1)$ и $\varphi(\delta C_2)$ связны по построению. Образ кограницы любой другой вершины связен, так как состоит из одного ребра, ибо все вершины $CB(n_1, n_2)$, кроме $C_1$ и $C_2$ --- висячие. Это значит, что пути между концами рёбер одной кограницы вообще не содержат рёбер других кограниц. Поэтому биекция $\varphi$ дружественная.
\end{proof}

Для доказательства необходимости в теореме \ref{3_criteria} нам понадобится утверждение \ref{st_connected}.

\begin{statement}
\label{st_connected}
Пусть $\varphi$ --- дружественная биекция между $CB(n_1, n_2)$ и неким деревом. Тогда $\varphi(\delta C_1)$ и $\varphi(\delta C_2)$ --- связные\footnote{Под связным множеством рёбер дерева мы подразумеваем множество рёбер дерева такое, что на любом пути между рёбрами этого множества лежат только рёбра этого множества.} множества.
\end{statement}

\begin{proof} Докажем от противного. Пусть $\varphi(\delta C_1)$ не является связным множеством. Тогда существует путь между концами рёбер, принадлежащих $\varphi(\delta C_1)$, содержащий хотя бы одно ребро $e$ не из $\varphi(\delta C_1)$. Рассмотрим $\varphi^{-1}(e)$.

Ясно, что $E(CB(n_1, n_2))$ состоит из:
\begin{itemize}
\item ребра $C_1C_2$,
\item висячих рёбер, выходящих из $C_1$,
\item висячих рёбер, выходящих из $C_2$.
\end{itemize}
Раз $\varphi^{-1}(e) \notin \delta C_1$, то $\varphi^{-1}(e)$ --- висячее ребро, смежное с $C_2$. Его второй конец --- это некая висячая вершина $Q$, находящаяся на чётном расстоянии от $C_1$. Множество, состоящее из $\varphi^{-1}(e)$, является кограницей вершины~$Q$.

Таким образом, путь между концами рёбер образа $\varphi(\delta C_1)$, содержащий ребро $e$,  содержит нечётное количество рёбер из $\varphi(\delta Q)$. Значит, образ $\varphi(\delta C_1)$ зацеплен с образом $\varphi(\delta Q)$. Противоречие.

Аналогично доказывается, что $\varphi(\delta C_2)$ --- связное множество.

\end{proof}

\begin{proof}[Доказательство необходимости в теореме \ref{3_criteria}]
Пусть существует дружественная биекция $\varphi: E(CB(n_1, n_2)) \to E(G)$. Тогда $\varphi(\delta C_1)$ и $\varphi(\delta C_2)$ связны согласно утверждению \ref{st_connected} и $\varphi(\delta C_1) \cap \varphi(\delta C_2) = \varphi(C_1C_2)$. Поддеревья $H_1$ и $H_2$, чьи множества рёбер --- это $\varphi(\delta C_1)$ и $\varphi(\delta C_2)$, соответственно, являются искомыми поддеревьями.
\end{proof}
        \subsection{Доказательство утверждения \ref{CB_n_4}}
            \graphicspath{{10/}}

Для доказательства утверждения \ref{CB_n_4} докажем следующее.

\begin{lemma}
\label{CB_n_4_lemma}
Пусть $n$ равно $2$, $3$, или $4$. Пусть в дереве $G$ количество рёбер не меньше $n$. Тогда в $E(G)$ можно выбрать два подмножества $E_1$ и $E_2$, которые пересекаются ровно по одному ребру, такие, что соответствующие им поддеревья связны, $|E_1| = |E(G)| - n + 1$, $|E_2| = n$.
\end{lemma}

Из леммы \ref{CB_n_4_lemma} и теоремы \ref{3_criteria} будет следовать, что деревья $CB(n, 2)$, $CB(n, 3)$ и $CB(n, 4)$ дружественны любому дереву. Для доказательства леммы \ref{CB_n_4_lemma} понадобится обозначение.

\begin{notation}[$\deg Q$]
Для вершины $Q$ обозначим через $\deg Q$ её степень.
\end{notation}

\begin{proof}[Доказательство леммы \ref{CB_n_4_lemma} для $n = 2$]
Обозначим через $e$ некоторое висячее ребро дерева $G$, а через $w$ --- ребро, имеющее с $e$ общую вершину. Возьмём 
$$E_1 := E(G)\setminus \{e\} \text{ и } E_2 := \{e, w\}.$$

Ясно, что поддеревья, соответствующие $E_1$ и $E_2$, связны.
\end{proof}

\begin{proof}[Доказательство леммы \ref{CB_n_4_lemma} для $n = 3$]
Обозначим через $P$ висячую вершину дерева $G$, а через $Q$ --- какую-нибудь из вершин, лежащих от $P$ на максимальном расстоянии. Очевидно, $Q$ --- висячая вершина. Обозначим через $Q'$ вершину, смежная с $Q$. Вершина $Q'$ не является висячей и не совпадает с $P$. Обозначим через $Q''$ вершину, смежную с $Q'$ и лежащую на пути из $Q'$ в $P$ (вершины $Q''$ и $P$ могут совпадать).

\begin{figure}
\begin{minipage}{0.49\linewidth}
\center{\includegraphics[height=5cm]{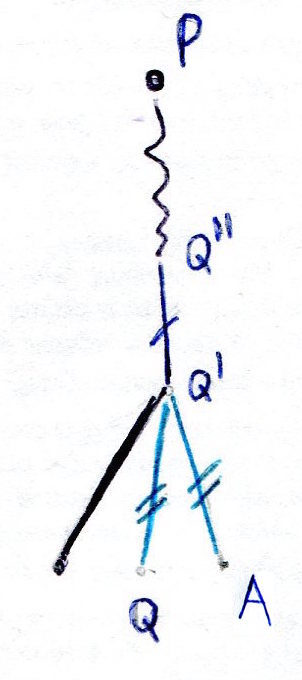}}
\caption{\label{fig:3-ge3}$n = 3, \deg Q' > 2$.}
\end{minipage}
\hfill
\begin{minipage}{0.49\linewidth}
\center{\includegraphics[height=5cm]{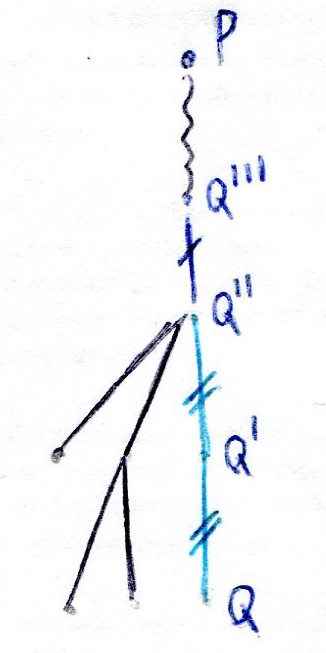}}
\caption{\label{fig:3-2}$n = 3, \deg Q' = 2$.}
\end{minipage}
\end{figure}

{\it Случай: $\deg Q' = 2$,} см. рис. \ref{fig:3-2}. В этом случае вершина $Q''$ не может быть висячей. Значит $Q''$ не совпадает с $P$. Обозначим через $Q'''$ вершину, смежную с $Q''$ и лежащую на пути из $Q''$ в $P$ (вершины $Q'''$ и $P$ могут совпадать). Возьмём 
$$E_1 := E(G) \setminus \{QQ', Q'Q''\}\text{ и }E_2 := \{QQ', Q'Q'', Q''Q'''\}.$$ 

Ясно, что поддеревья, соответствующие $E_1$ и $E_2$, связны.

{\it Случай: $\deg Q' > 2$,} см. рис. \ref{fig:3-ge3}. Обозначим через $A$ вершину, смежную с $Q'$, и не совпадающую с вершинами $Q$ и $Q''$. Возьмём 
$$E_1 := E(G) \setminus \{QQ', AQ'\}\text{ и }E_2 := \{QQ', AQ', Q'Q''\}.$$ 

Поддерево, соответствующее $E_2$, очевидно, связно. Чтобы доказать, что поддерево, соответствующее $E_1$, связно, достаточно убедиться в том, что вершина $A$ --- висячая. Действительно, если вершина $A$ не является висячей, тогда вершина, смежная с $A$ и не совпадающая с $Q'$, находится от вершины $P$ на большем расстоянии, чем вершина $Q$. Противоречие.
\end{proof}

\begin{proof}[Доказательство леммы \ref{CB_n_4_lemma} для $n = 4$]
Обозначим через $P$ висячую вершину дерева $G$, а через $Q$ --- какую-нибудь из вершин, лежащих от $P$ на максимальном расстоянии. Очевидно, $Q$ --- висячая вершина. Обозначим через $Q'$ вершину, смежная с $Q$. Вершина $Q'$ не является висячей и не совпадает с $P$. Обозначим через $Q''$ вершину, смежную с $Q'$ и лежащую на пути из $Q'$ в $P$ (вершины $Q''$ и $P$ могут совпадать). Рассуждением от противного, которое повторяет проведённое в разборе случая «$n = 3$, $\deg Q' > 2$», доказывается, что все вершины, кроме $Q''$, смежные с $Q'$, являются висячими. 

\begin{figure}
\begin{minipage}{0.49\linewidth}
\center{\includegraphics[height=5cm]{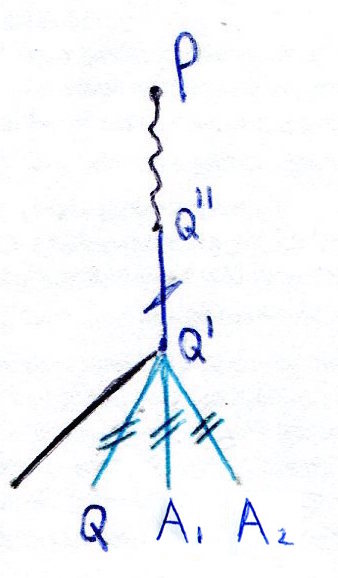}}
\caption{\label{fig:4-ge4}$n = 4, \deg Q' > 3$.}
\end{minipage}
\hfill
\begin{minipage}{0.49\linewidth}
\center{\includegraphics[height=5cm]{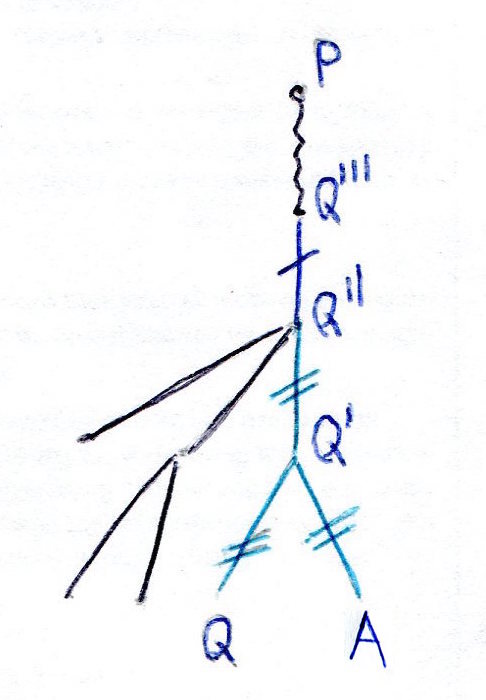}}
\caption{\label{fig:4-3}$n = 4, \deg Q' = 3$.}
\end{minipage}
\end{figure}

{\it Случай: $\deg Q' > 3$,} см. рис. \ref{fig:4-ge4}.
Обозначим через $A_1$, $A_2$ вершины, смежные с $Q'$ и не совпадающие с $Q$ и $Q''$.  Возьмём 
$$E_1 := E(G)\setminus \{QQ', A_1Q', A_2Q'\}\text{ и }E_2 := \{QQ', A_1Q', A_2Q', Q'Q''\}.$$ 

Поддерево,
соотвествующее $E_2$, очевидно, связно. Вершины $A_1$, $A_2$ являются висячими, значит, и поддерево, соответствующее $E_1$, связно.

{\it Случай: $\deg Q' \leqslant 3$}. Если $Q''$ --- висячая вершина и $\deg Q' \leqslant 3$, то в дереве $G$ не более трёх рёбер, что противоречит условию леммы. Значит, при $\deg Q' \leqslant 3$ вершина $Q''$ не является висячей. Следовательно вершина $Q''$ не совпадает с $P$. Обозначим через $Q'''$ вершину, смежную с $Q''$ и лежащую на пути из $Q''$ в $P$ (вершины $Q'''$ и $P$ могут совпадать).

{\it Подслучай: $\deg Q' = 3$,} см. рис. \ref{fig:4-3}.
Обозначим через $A$ вершину, смежную c $Q'$ и не совпадающую с вершинами $Q$ и $Q''$. Возьмём 
$$E_1 := E(G)\setminus \{QQ', AQ', Q'Q''\}\text{ и }E_2 := \{QQ', AQ', Q'Q'', Q''Q'''\}.$$ 

Ясно, что поддеревья, соответствующие $E_1$ и $E_2$, связны.

\begin{figure}
\begin{minipage}{0.25\linewidth}
\center{\includegraphics[height=5cm]{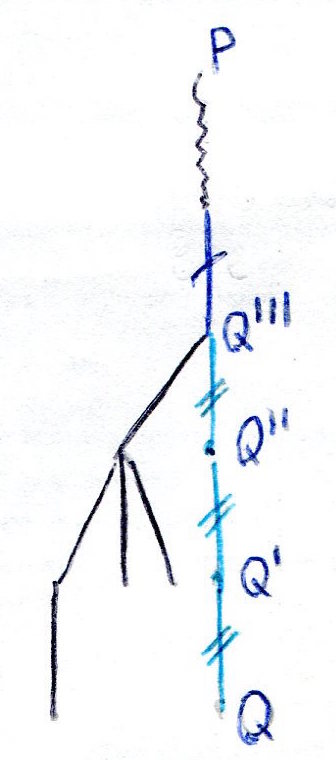}}
\caption{\label{fig:4-2-2}$n = 4$, $\deg Q' = 2$, $\deg Q'' = 2$.}
\end{minipage}
\hfill
\begin{minipage}{0.33\linewidth}
\center{\includegraphics[height=5cm]{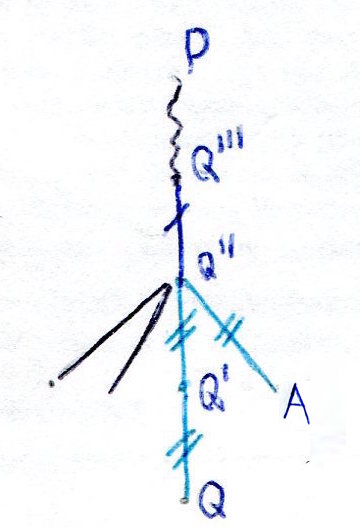}}
\begin{minipage}{0.78\linewidth}
\caption{\label{fig:4-2-ge3-1}$n = 4$, $\deg Q' = 2$, $\deg Q'' > 2$, $A$~---~висячая.}
\end{minipage}
\end{minipage}
\hfill
\begin{minipage}{0.33\linewidth}
\center{\includegraphics[height=5cm]{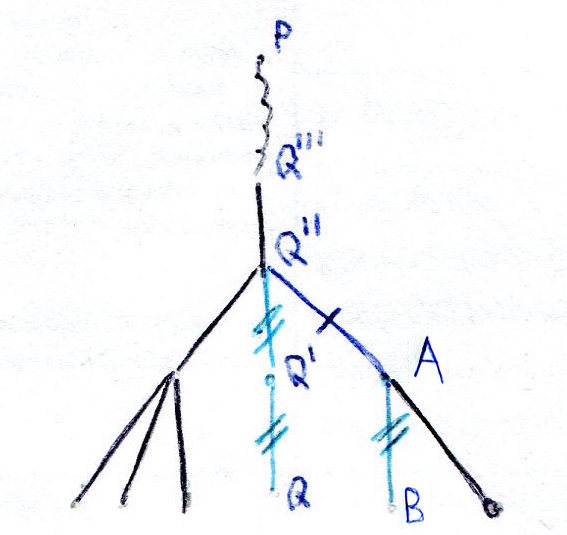}}
\begin{minipage}{0.78\linewidth}
\caption{\label{fig:4-2-ge3-ge2}$n = 4$, $\deg Q' = 2$, $\deg Q'' > 2$, $A$~---~не~висячая.}
\end{minipage}
\end{minipage}
\end{figure}

{\it Подслучай: $\deg Q' = 2$, $\deg Q'' = 2$,} см. рис. \ref{fig:4-2-2}.
В этом случае $Q'''$ не может быть висячей вершиной. Обозначим через $e$ ребро, смежное с $Q'''$, не совпадающее с $Q'''Q''$. Возьмём 
$$E_1 := E(G)\setminus \{QQ', Q'Q'', Q''Q'''\}\text{ и }E_2 := \{QQ', Q'Q'', Q''Q''', e\}.$$

Ясно, что поддеревья, соответствующие $E_1$ и $E_2$, связны.

{\it Подслучай: $\deg Q' = 2$, $\deg Q'' > 2$}. Обозначим через $A$ вершину, смежную с $Q''$, и не совпадающую с вершинами $Q'''$ и $Q'$.
Если $A$ --- висячая, см. рис. \ref{fig:4-2-ge3-1}, возьмём 
$$E_1 := E(G)\setminus \{QQ', Q'Q'', AQ''\}\text{ и }E_2 := \{QQ', Q'Q'', AQ'', Q''Q'''\}.$$

Ясно, что поддеревья, соответствующие $E_1$ и $E_2$, связны.

Если $A$ не является висячей, то обозначим через $B$ вершину, смежную с $A$ и не совпадающую с $Q''$, см. рис. \ref{fig:4-2-ge3-ge2}. Возьмём 
$$E_1 := E(G)\setminus \{QQ', Q'Q'', BA\}\text{ и }E_2 := \{QQ', Q'Q'', BA, AQ''\}.$$

Поддерево, соответствующее $E_2$, очевидно, связно. Чтобы доказать, что поддерево, соответствующее $E_1$, связно, достаточно убедиться в том, что вершина $B$ --- висячая. Действительно, если вершина $B$ не является висячей, тогда вершина, смежная с $B$ и не совпадающая с $A$, находится от вершины $P$ на большем расстоянии, чем вершина $Q$. Противоречие.

\end{proof}

    \section{Благодарности}
        \noindent
Спасибо А.\,Б.~Скопенкову, научному руководителю, за плодотворные обсуждения.

\noindent
Спасибо Г.\,Г.~Гусеву, рецензенту, за полезный отзыв.

\noindent
Спасибо С.\,Я.~Аввакумову, коллеге, за ценные замечания.

\noindent
Спасибо А.\,В.~Колодзею и Ф.\,М.~Алексееву за поддержку.

\end{document}